\documentclass[12pt,twoside,a4paper]{article}
\usepackage{graphicx}
\usepackage[english]{babel}
\usepackage{amsmath}
\usepackage{amsthm}
\usepackage{amscd}
\usepackage{pb-diagram}
\usepackage{comment}
\usepackage{amssymb}
\usepackage{subcaption}
\usepackage{color}
\usepackage{enumitem}
\usepackage{faktor}
\usepackage[normalem]{ulem}  

\usepackage{amsmath, amssymb, amsthm}  
\usepackage{mathtools}  
\usepackage{bm}  
\usepackage{bbm}  
\usepackage{physics}  

\usepackage{tikz}
\usetikzlibrary{arrows, automata, positioning}

\usepackage{graphicx}  
\usepackage{float}  

\usepackage{geometry}  
\usepackage{times}  
\usepackage{setspace}  

\usepackage{hyperref}
\hypersetup{
    colorlinks=true,
    linkcolor=blue,
    citecolor=blue,
    urlcolor=blue
}

\theoremstyle{definition}
\newtheorem{definition}{Definition}
\theoremstyle{plain}
\newtheorem{theorem}{Theorem}
\newtheorem{example}{Example}

\newtheorem{lemma}{Lemma}

\theoremstyle{remark}
\newtheorem{remark}{Remark}

\usepackage{authblk}

\usepackage{algorithm}
\usepackage{algorithmicx}
\usepackage{algpseudocode}

\newcommand{\mo}{\operatorname{mo}}
\newcommand{\cl}{\operatorname{cl}}

\newcommand{\img}{\operatorname{im}}

\newcommand{\esol}{\operatorname{eSol}}

\def\mathobj#1{\mbox{$#1$}}

\def\ZZ{\mathobj{\mathbb{Z}}}

\newcommand{\dom}{\operatorname{dom}}
\newcommand{\Inv}{\operatorname{Inv}}

\def\cV{\text{$\mathcal V$}}

\def\articletheorems{
\newtheorem{thm}{Theorem}[section]

\newtheorem{prop}[thm]{Proposition}

}
\articletheorems

\renewenvironment{proof}{{\bf Proof:\ }}{\qedsymbol}

\begin{document}

\title{Persistent homology of Morse decomposition in Markov chains based on combinatorial multivector fields}
 \author{Donald Woukeng}
\affil{Division of Computational Mathematics, Faculty of Mathematics and Computer
Science, Jagiellonian University, ul. St. Lojasiewicza 6, Krakow, 30-348, Poland, donald.woukeng@aims.ac.rw}
\date{}
\maketitle
 

\begin{abstract}
    In this paper, we introduce a novel persistence framework for Morse decompositions in Markov chains using combinatorial multivector fields. Our approach provides a structured method to analyze recurrence and stability in finite-state stochastic processes. In our setting filtrations are governed by transition probabilities rather than spatial distances. We construct multivector fields directly from Markov transition matrices, treating states and transitions as elements of a directed graph. By applying Morse decomposition to the induced multivector field, we obtain a hierarchical structure of invariant sets that evolve under changes in transition probabilities. This structure naturally defines a persistence diagram, where each Morse set is indexed by its topological and dynamical complexity via homology and Conley index dimensions.

\end{abstract}
\section{Introduction}
\subsection{Motivation}
The study of dynamics has seen a growing interaction between topological methods and discrete combinatorial approaches, especially in understanding recurrence and stability in finite-state models. One of the key frameworks in this domain is Morse decomposition, which provides a hierarchical structure for understanding long-term system behavior. Persistent homology has been instrumental in tracking topological features across scales, but its direct application to Morse decompositions has remained a challenge.

Several works have addressed the persistence of topological invariant for dynamics data or dynamical systems, it is the case for example of  \cite{dey2019persistent,dey2020persistence,dey2022persistence} where the study of persistence is done for Morse decomposition, Conley index and Conley-Morse graph. However, a key missing component in these works is a stability theorem that ensures that Morse decompositions remain robust under small perturbations in the system's transition dynamics. In particular, no prior work has established a formal stability result for Morse set persistence in Markov chains. Our work addresses this gap by providing a rigorous foundation for the stability of Morse decompositions in combinatorial multivector fields constructed from Markov chains.
Markov chains are widely used to model stochastic processes in various real-world domains, including biological networks , financial markets, climate systems, and reinforcement learning. In these applications, understanding long-term behavior and recurrent structures is crucial for predicting trends, identifying stable states, and classifying different system dynamics. Some traditional methods for analyzing Markov chains focus on stationary distributions, spectral properties, and absorbing states \cite{fernandez2025quasi,ben2024representation,han2024classification}, but they often lack a topological and hierarchical perspective on the structure of state transitions. Some recent work have been done for the analysis of Markov chains using classical persistence \cite{le2022persistent,tymochko2021classifying,donato2013decimation}, but they do not take into account the potential dynamics hidden in the Markov chains.

Our method introduces a new classification approach for Markov chains by applying persistent homology to Morse decompositions, allowing us to track how recurrent structures change as transition probabilities vary. By constructing persistence diagrams indexed by homology and Conley indices, we provide a robust way to classify Markov chains based on their intrinsic recurrence structures rather than just numerical transition probabilities. 
\subsection{Overview of the result}
In this paper, we develop a persistence framework for Morse decompositions in Markov chains using combinatorial multivector fields. Our main contributions is as follows:
\begin{theorem}[Stability of Morse Set Persistence Diagrams]
Let \( P \) and \( P' \) be two transition matrices of a Markov process with n states such that:

\[
\| P - P' \|_{\infty} < \delta.
\]

Then, the bottleneck distance between their Morse set persistence diagrams satisfies:

\begin{equation}
d_B(D(P), D(P')) < C \delta,
\end{equation}

where \( C \) depends only on the structure of the Markov process.
\end{theorem} 
 \section{Preliminaries}

\subsection{Relations and Posets}

Let $X$ be a set. A \emph{binary relation} on $X$ is defined as a subset $R \subset X \times X$. We use the standard notation $xRy$ to indicate that $(x,y) \in R$.

A relation $\leq$ that is reflexive, antisymmetric, and transitive is called a \emph{partial order}, and the pair $(X, \leq)$ is known as a \emph{poset}. A subset $A \subset X$ of a poset $(X, \leq)$ is called an \emph{upper set} if $\{ z \in X  \mid \exists_{x\in A} x \leq z \} \subset A$. Similarly, $A \subset X$ is a \emph{lower set} if $\{ z \in X  \mid \exists_{x\in A} z \leq x \} \subset A$.

A subset $A\subset X$ is said to be \emph{convex} in a poset $(X, \leq)$ if for every $x,z\in A$ and $y\in X$ satisfying $x \leq y \leq z$, it follows that $y \in A$. Equivalently, a convex set is the intersection of a lower set and an upper set.

A relation $R$ is called an \emph{equivalence relation} if it satisfies reflexivity, symmetry, and transitivity. The \emph{equivalence class} of an element $x$ under $R$ is given by $[x]_{R}:=\{y \in X \mid x R y\}$.

A \emph{partition} of $X$ is a collection $\mathcal{V}$ of non-empty subsets of $X$ such that for every distinct $A, B \in \mathcal{V}$, we have $A \cap B = \emptyset$ and $\bigcup \mathcal{V} = X$. A partition $\mathcal{V}$ induces an equivalence relation $R$ defined by $xRy$ if there exists some $A \in \mathcal{V}$ such that $x, y \in A$. We denote the equivalence class of a point $x \in X$ under partition $\mathcal{V}$ as $[x]_{\mathcal{V}}$.

\subsection{Topological Spaces}

Given a topology $\mathcal{T}$ on $X$, the pair $(X, \mathcal{T})$ is referred to as a \emph{topological space}. When the topology $\mathcal{T}$ is understood, we may simply refer to $X$ as a topological space. The interior, closure, and boundary of a subset $A \subset X$ with respect to $\mathcal{T}$ are denoted by $\operatorname{int}_{\mathcal{T}} A$, $\operatorname{cl}_{\mathcal{T}} A$, and $\operatorname{bd}_{\mathcal{T}} A$, respectively. The \emph{mouth} of $A$ is defined as $\operatorname{mo}_{\mathcal{T}} A := \operatorname{cl}_{\mathcal{T}} A \setminus A$. When the topology is clear from context, we may use the shorthand $\operatorname{cl}_X A$ or omit the subscript altogether.

For a subset $Y \subset X$, the induced topology on $Y$ is given by $\mathcal{T}_Y := \{U \cap Y \mid U \in \mathcal{T}\}$. The closure of $A \subset Y$ with respect to $\mathcal{T}_Y$ is denoted as $\operatorname{cl}_{\mathcal{T}_Y} A$ or simply $\operatorname{cl}_Y A$. The same applies to the interior, boundary, and mouth.

A subset $A \subset X$ is called \emph{locally closed} if every point $x \in A$ has a neighborhood $U$ in $X$ such that $A \cap U$ is closed in $U$.

\begin{prop}\cite[Problem 2.7.1]{En1989}. 
A subset $A$ of a topological space $X$ satisfies the following equivalent conditions:
\begin{enumerate}[label=(\roman*)]
    \item $A$ is locally closed,
    \item The mouth $\operatorname{mo} A = \operatorname{cl} A \setminus A$ is closed,
    \item $A$ can be expressed as the difference of two closed sets in $X$,
    \item $A$ is the intersection of an open set and a closed set in $X$.
\end{enumerate}
Moreover, the finite intersection of locally closed sets remains locally closed.
\end{prop}

A topological space is called a $T_0$ space if for every pair of distinct points in $X$, at least one of them has a neighborhood that does not contain the other. We are particularly interested in finite $T_0$ spaces, which, by Alexandroff’s theorem \cite{Alexandroff_ftop}, can be associated with posets.

\begin{thm}\label{thm:finite_top} \cite{Alexandroff_ftop}
For a finite poset $(P, \leq)$, the collection of all upper sets of $\leq$ forms a $T_0$ topology $\mathcal{T}_{\leq}$ on $P$. Conversely, for any finite $T_0$ space $(X, \mathcal{T})$, defining $x \leq_{\mathcal{T}} y$ by $x \in \operatorname{cl}_{\mathcal{T}} \{ y \}$ results in a partial order on $X$. These two correspondences are mutually inverse.
\end{thm}

This equivalence allows all topological properties of a finite $T_0$ space to be expressed in terms of posets and vice versa. In particular, given a finite $T_0$ space $(X, \mathcal{T})$ with an associated poset $(X, \leq)$, the closure and boundary of $A \subset X$ can be described as:
\begin{align}
    \operatorname{cl}_{\mathcal{T}} A &= \{x \in X \mid \exists_{a \in A} x \leq a\}, \label{eq:closure_poset}\\
    \operatorname{bd}_{\mathcal{T}} A &= \{x \in X \mid \exists_{a \in A, b \in X \setminus A} x \leq a \text{ and } x \leq b\}. \label{eq:boundary_poset}
\end{align}

In the context of finite topological spaces, the notions of local closedness and convexity coincide \cite[Proposition 1.4.10]{lipinski_phd}. Thus, a subset is locally closed in a $T_0$ topology if and only if it is convex in the corresponding poset structure. Throughout this paper, we use these terms interchangeably, favoring “locally closed” in topological discussions and “convex” in combinatorial or algorithmic contexts.


\subsection{Combinatorial multivector fields}

 All of the definitions in this subsection can be found in \cite{lipinski2019conley}.

Let X be a finite topological space. A \emph{combinatorial multivector} or briefly a \emph{multivector} is a
locally closed subset $V \subset X$. 
A \emph{combinatorial multivector field} on $X$, or briefly a \emph{multivector field}, is a partition $\mathcal{V}$ of $X$ into multivectors.

Since $\cV$ is a partition, we can denote by $[x]_{\mathcal{V}}$ the unique multivector in $\cV$ that contains $x\in X$. 
If the multivector field $\cV$ is clear from the context, we write briefly $[x]$.
We say that a multivector $V\in\cV$ is \emph{critical} if the relative singular homology $H(\cl V, \mo V )$ is non-trivial.
A multivector $V$ which is not critical is called \emph{regular}. 
We say that a set $A \subset X$ is \emph{$\cV$-compatible} if for every $x \in X$ either $[x] \cap A = \emptyset$ or $[x] \subset A$.

Multivector field $\mathcal{V}$ on $X$ induces a  multivalued map $\Pi_{\mathcal{V}} : X \multimap X$ given by 
\begin{align}\label{eq:piv}
    \Pi_{\mathcal{V}}(x)=[x]_{\mathcal{V}}\cup \cl x~.
\end{align}

We consider a combinatorial dynamical system given by the iterates of~$\Pi_\cV$.

A \emph{solution} of a combinatorial dynamical system $\Pi_\cV:X\multimap X$ in $A\subset X$ is a partial map $\varphi:\mathbb{Z}\nrightarrow A$ whose domain, denoted $\dom \varphi$,  is a $\mathbb{Z}$-interval and for any $i,i+1\in \dom \varphi$ the inclusion $\varphi(i+1)\in \Pi_\cV(\varphi(i))$ holds. Let us denote by $\textup{Sol}(A)$ the set of all solutions $\varphi$ such that $\img\varphi\subset A$. $\textup{Sol}(X)$ is the set of all solution of $\Pi_\cV$.
If $\dom\varphi$ is a bounded interval then we say that $\varphi$ is a \emph{path}.
If $\dom\varphi=\ZZ$ then $\varphi$ is a \emph{full solution}. 

A full solution $\varphi : \mathbb{Z} \rightarrow X$ is \emph{left-essential} (respectively \emph{right-essential})
if for every regular $x \in \img\varphi$ the set $\{ t \in \mathbb{Z}\mid \varphi(t) \notin [x]_{\mathcal{V}} \}$ is left-infinite (respectively right-infinite). 
We say that $\varphi$ is \emph{essential} if it is both left- and right-essential.
The collection of all essential solutions $\varphi$ such that $\img\varphi\subset A$ is denoted by $\textup{eSol}(A)$.

The \emph{invariant} part of a set $A\subset X$ is 
$\Inv A := \bigcup \{\img \varphi\mid \varphi\in\esol(A)\}$.
In particular, if $\textup{Inv} A = A$ we say that $A$ is an \emph{invariant set} for a multivector field $\mathcal{V}$.

A closed set $N\subset X$ \emph{isolates} invariant set $S \subset N$ if the following conditions are satisfied:
\begin{enumerate}[label=(\roman*)]
    \item every path in $N$ with endpoints in $S$ is a path in $S$,
    \item $\Pi_{\mathcal{V}}(S) \subset N$.
\end{enumerate}
In this case, $N$ is an \emph{isolating set} for $S$. 
If an invariant set $S$ admits an isolating set then we say that $S$ is an \emph{isolated invariant set}.
The \emph{homological Conley index} of an isolated invariant set $S$ is defined as $\operatorname{Con}(S):=H(\operatorname{cl} S,\mo S)$.

Let $A \subset X$.
By $\bigl\langle A \bigl\rangle_{\mathcal{V}}$ we denote the intersection of 
all locally closed and $\cV$-compatible sets in $X$ containing $A$.
We call this set the $\mathcal{V}$-\emph{hull} of $A$. 
The combinatorial $\alpha$-\emph{limit set} and $\omega$-\emph{limit set} for a full solution $\varphi$ are defined as
\begin{align*}
    & \alpha(\varphi) := \Bigl\langle \bigcap\limits_{t \in \mathbb{Z}^-}\varphi((-\infty,t]) \Bigl\rangle_\mathcal{V}\ , \\
    & \omega (\varphi) := \Bigl\langle \bigcap\limits_{t \in \mathbb{Z}^+}\varphi([t, \infty))  \Bigl\rangle_\mathcal{V}\ .
\end{align*}

Let $S\subset X$ be a $\cV$-compatible, invariant set.
Then, a finite collection $\mathcal{M}=\{M_p\subset S\mid p\in\mathbb{P}\}$ is called a \emph{Morse decomposition} of $S$ if there exists a finite poset $(\mathbb{P},\le)$ such that the following conditions are satisfied:
\begin{enumerate}[label=(\roman*)]
    \item $\mathcal{M}$ is a family of mutually disjoint, isolated invariant subsets of $S$,
    \item for every $\varphi\in\esol(S)$ either $\img\varphi \subset M_r$  for an $r \in \mathbb{P}$
or there exist $p, q \in \mathbb{P}$ such that $q > p$,    $\alpha(\varphi)\subset M_q \text{, and } \omega(\varphi)\subset M_p$.
 \end{enumerate}
We refer to the elements of $\mathcal{M}$ as \emph{Morse sets}.

Let \( \cV \) and \( \cV' \) be two multivector fields on the same underlying space. We say that \( \cV \) is a \textbf{coarsening} of \( \cV' \) (see \cite{dey2022tracking}, denoted \( \cV' \preceq \cV \), if every multivector in \( \cV \) is a union of one or more multivectors in \( \cV' \). Formally, 
\[
\forall~ V \in \cV, \quad \exists \{W_i\} \subseteq \cV' \text{ such that } V = \bigcup_i W_i.
\]
Coarsening corresponds to merging finer multivectors into larger ones, effectively reducing the resolution of the multivector field.
\subsection{Markov Chains}
The defintions on this subsection can be found in \cite[Chapter 2]{tolver2016introduction}
\subsubsection{Definition of a Markov Chain}

A discrete-time Markov chain is a stochastic process \(\{X(n)\}_{n\in \mathbb{N}_0}\) taking values in a finite or countable state space \(S\). The process satisfies the Markov property, meaning that the probability of transitioning to the next state depends only on the present state and not on the past history:

\begin{equation}
    P(X(n+1) = j \mid X(n) = i, X(n-1) = i_{n-1}, \dots, X(0) = i_0) = P_{i,j},
\end{equation}

for all \(i, j, i_{n-1}, \dots, i_0 \in S\) and \(n \in \mathbb{N}_0\).

A Markov chain is called \textit{time-homogeneous} if the transition probabilities do not depend on \(n\), i.e.,

\begin{equation}
    P(X(n+1) = j \mid X(n) = i) = P_{i,j}.
\end{equation}

The dynamics of a discrete-time Markov chain on a finite state space is fully characterized by its \textit{transition matrix} \(P = (P_{i,j})\), where each entry \(P_{i,j}\) represents the probability of transitioning from state \(i\) to state \(j\):

\begin{equation}
    P = 
    \begin{bmatrix}
        P_{1,1} & P_{1,2} & \dots & P_{1,N} \\
        P_{2,1} & P_{2,2} & \dots & P_{2,N} \\
        \vdots & \vdots & \ddots & \vdots \\
        P_{N,1} & P_{N,2} & \dots & P_{N,N}
    \end{bmatrix}.
\end{equation}

Each row of \(P\) sums to 1:

\begin{equation}
    \sum_{j \in S} P_{i,j} = 1, \quad \forall i \in S.
\end{equation}

The probability of the system being in a specific state at step \(n\) is given by the vector \(\mathbf{p}(n)\), which evolves according to the equation:

\begin{equation}
    \mathbf{p}(n) = \mathbf{p}(0) P^n.
\end{equation}

where \(\mathbf{p}(0)\) is the initial probability distribution over states.

\subsubsection{Graph Representation of a Markov Chain}

A Markov chain can be represented as a \textit{directed graph}, called a \textit{transition diagram}. In this representation:

- The nodes (vertices) correspond to the states of the Markov chain.
- A directed edge \(i \to j\) exists if \(P_{i,j} > 0\), meaning there is a positive probability of transitioning from state \(i\) to state \(j\).
- The edges are labeled by the corresponding transition probabilities \(P_{i,j}\).

For example, the transition diagram of a three-state Markov chain with the transition matrix:

\begin{equation}
    P = 
    \begin{bmatrix}
        1/3 & 1/3 & 1/3 \\
        1/2 & 1/2 & 0 \\
        1 & 0 & 0
    \end{bmatrix}
\end{equation}

is represented as:

\begin{center}
\begin{tikzpicture}
    \node (1) at (0,0) {1};
    \node (2) at (3,1) {2};
    \node (3) at (3,-1) {3};

    \draw[->] (1) to[out=30,in=150] node[above] {\(\frac{1}{3}\)} (2);
    \draw[->] (1) to[out=-30,in=210] node[below] {\(\frac{1}{3}\)} (3);
    \draw[->] (1) to[loop left] node[left] {\(\frac{1}{3}\)} (1);

    \draw[->] (2) to[out=-45,in=45] node[right] {\(\frac{1}{2}\)} (1);
    \draw[->] (2) to[loop right] node[right] {\(\frac{1}{2}\)} (2);

    \draw[->] (3) to[out=150,in=-150] node[left] {1} (1);
\end{tikzpicture}
\end{center}

where the directed edges indicate possible transitions and their associated probabilities.
 
\section{Morse decomposition for Markov chains}
We introduce the construction of \textit{multivector fields} from Markov chains, explore the notion of \textit{coarsening}, and establish the \textit{Morse decomposition} of a multivector field. These concepts provide a structured way to analyze the dynamics of state transitions within a Markov process and extract meaningful topological features from the induced flow structure.

\subsection{Multivector field construction}
A \textit{multivector field} is constructed by merging \textit{vertices and edges} based on transition probabilities:
\begin{itemize}
    \item Each \textit{state} in the Markov chain is represented as a \textit{vertex}.
    \item Each \textit{directed edge} represents a transition between states.
    \item \textit{Edges and vertices are merged} into \textit{multivectors} if the transition probability is \textit{less than or equal a chosen threshold} \( \gamma \).
    \item Any element (vertex or edge) \textit{that is not merged} remains a \textit{singleton multivector}.
\end{itemize}

This construction ensures that \textit{all vertices and edges are included} in the multivector field, forming a \textit{directed partition of the state space and the edges that have non-zero transition probability}. The process is \textit{iterative}, merging multivectors when they \textit{overlap}, ensuring a well-defined decomposition.

A formal \textit{algorithm} is presented here, defining the precise steps to construct the multivector field. We will follow a method similar to the one use in \cite{Woukeng_2024, cote2025data}.

\begin{algorithm}
\caption{Construction of the Multivector Field from a Markov Chain}
\begin{algorithmic}[1]
\Require Transition matrix \( P \), threshold \( \gamma \)
\Ensure Multivector field \( V \)

\State Initialize \( V \gets \emptyset \) \Comment{Start with an empty set of multivectors}
\State Define \( S \) as the set of all states (nodes) in the Markov chain
\State Define \( E \) as the set of all directed edges \( (N_i, N_j) \) where \( P(N_i \to N_j) > 0 \)

\Comment{Step 1: Initialize all vertices and edges as separate multivectors}
\ForAll{ \( N_i \in S \) }
    \State \( V \gets V \cup \{ \{N_i\} \} \)
\EndFor
\ForAll{ \( (N_i, N_j) \in E \) with \( i < j \) }
    \State \( V \gets V \cup \{ \{(N_i, N_j)\} \} \) \Comment{Avoid duplicate bidirectional edges}
\EndFor

\Comment{Step 2: Merge vertices with edges if transition probability is below \( \gamma \)}
\ForAll{ \( (N_i, N_j) \in E \) with \( i < j \) }
    \If{ \( P(N_i \to N_j) \le \gamma \)}
        \State Merge \( N_i \) and the edge \( (N_i, N_j) \) into one multivector
        \State \( V \gets V \setminus \{ \{N_i\}, \{(N_i, N_j)\} \} \cup \{ \{N_i, (N_i, N_j)\} \} \)
    \EndIf
    \If{\( P(N_j \to N_i) \le \gamma \) }
        \State Merge \( N_j \) and the edge \( (N_i, N_j) \) into one multivector
        \State \( V \gets V \setminus \{\{N_j\}, \{(N_i, N_j)\} \} \cup \{ \{ N_j, (N_i, N_j)\} \} \)
    \EndIf
\EndFor

\Comment{Step 3: Merge overlapping multivectors to maintain the partition}
\ForAll{ pairs \( V_k, V_m \in V \) }
    \If{ \( V_k \cap V_m \neq \emptyset \) }
        \State Merge \( V_k \) and \( V_m \) into a single multivector
        \State \( V \gets V \setminus \{ V_k, V_m \} \cup \{ V_k \cup V_m \} \)
    \EndIf
\EndFor

\State \Return \( V \) \Comment{Return the Multivector field V}
\end{algorithmic}
\end{algorithm}

\begin{theorem}
Let \( P \) be the transition matrix of a finite Markov process with state space \( S \), and let \( \cV\gamma \) be the multivector field constructed using the threshold \( \gamma \). Then:

\begin{itemize}

    \item   The algorithm always produces a valid multivector field \( \cV\gamma \), where each state and transition appears exactly once.
    
    \item   The sequence of multivector fields \( \cV{\gamma_1}, \cV{\gamma_2}, \dots \) forms a decreasing filtration:
    \[
    \cV{\gamma_1} \preceq \cV{\gamma_2} \preceq \dots
    \]
    for \( \gamma_1 < \gamma_2 \), meaning that as \( \gamma \) increases.
    .
\end{itemize}
\end{theorem}
 
\begin{proof}
    \begin{itemize}
        \item We consider our finite space being the state space of our Markov process plus edges between states, if the probability of going from one state to another is non-zero. We will denote that space X.
        \[
        X=\bigcup\{N_i\}\cup\{(N_i,N_j)\}\quad \text{if}\quad P(N_i\to N_j)>0~or~P(N_j\to N_i)>0,~i<j
        \]
        In Algorithm 1, from line 4 to 9 by adding all the elements of our space $X$, vertices representing states, and edges representing connection between them with non-zero transition probability. From line 10 to 19, we merge edges and vertices of our space according to the fix given threshold $\gamma$, and the last part from line 20 to 24, we ensure that what we created is  partition, by merging together elements that have non empty intersection with other elements. We then have a partition of our space $X$, and since we only deal with edges and vertices, we are sure to always have the local closedness property. Since V returned at the line 26 is a partition of our space X, into locally closed subset, V is then a multivector field.

        \item The second part of the theorem follows naturally from the construction of the multivector field, since it is done by merging edge to vertices for a certain threshold. If we have a bigger threshold, it is only natural that the multivectors will contain at least all the element contains, for a lower threshold, and maybe we will have multivector that are just union of multivectors in a for a lower threshold. 
    \end{itemize}
\end{proof}

\begin{example}
Consider a Markov process with three states \( S = \{N_1, N_2, N_3\} \) and the transition matrix:

\[
P =
\begin{bmatrix}
0.5 & 0.17 & 0.33 \\
0.17 & 0.6 & 0.23 \\
0.15 & 0.15 & 0.7
\end{bmatrix}.
\]

We choose a threshold \( \gamma = 0.2 \) for merging.

\begin{enumerate}
    \item \textbf{Step 1: Initialize Multivectors}\\
    Each state starts as its own singleton:  
    \[
    V1 = \{N_1\}, \quad V2 = \{N_2\}, \quad V3 = \{N_3\},
    \]\[
    V_4=\{(N_1,N_2)\},\quad V_5=\{(N_1,N_3)\}, \quad V_6=\{(N_2,N_3)\}.
    \]
    \item \textbf{Step 2: Merge Based on Transition Probabilities}\\
    - Since \( P_{12} = 0.17 \le \gamma \), we merge $V_4$ into \( V_1 \).
     Since \( P_{31} = 0.15 \le \gamma \), we merge $V_5$ into \( V_3 \).
    Since \( P_{32} = 0.15 \le \gamma \), we merge \(V_6 \) into \( V_3 \).  - Since \( P_{21} = 0.17 \le \gamma \), we merge \(V_4 \) into \( V_2 \).

    The new candidates are:
    \[
    V_1^* = \{N_1, (N_1, N_2)\}, \quad V_2^* = \{N_2, (N_1,N_2)\}, \quad V_3^* = \{N_3, (N_2, N_3), (N_1,N_3)\}.
    \]

    \item \textbf{Step 3: Partition and Merge Overlapping Multivectors}\\
    Since \( (N_1,N_2) \) appears in \( V_1^* \) and \( V_2^* \), we merge them:
    \[
    V_3^* = \{N_3, (N_1, N_3), (N_2, N_3)\}, \quad V_1^* = \{N_1,N_2, (N_1,N_2)\}.
    \]

    \item \textbf{Step 4: Final Multivector Field}\\
    The final set of multivectors is:
    \[
    \cV = \{V_1^*, V_3^*\}.
    \]
\end{enumerate}
\end{example}

\subsection{Morse Decomposition in Multivector Fields}
 Once the multivector field is constructed, we analyze its \textit{global structure} through Morse decompositions. A \textit{Morse set} is defined as a \textit{strongly connected component} (see\cite{lipinski2019conley}) in the \textit{graph induced by the multivector field (M-graph)}. These Morse sets represent \textit{regions of recurrence and stability} in the system.

\begin{itemize}
    \item The Morse sets are computed from the \textit{M-graph}, where \textit{each multivector is a node}, and \textit{edges exist based on the mouth condition} (a multivector \( M \) transitions to another \( M' \) if \( M' \) contains an element in \( M \)'s mouth).
    \item The resulting Morse decomposition provides a \textit{hierarchical structure} of stability regions.
    \item Morse sets can be \textit{persistent} under coarsening, allowing us to track their evolution through different scales.
\end{itemize}

A key property of this approach is that Morse decompositions naturally arise from the multivector field and do not require predefined distance metrics, making them applicable to discrete-state systems such as Markov chains.

A Morse set \( M \) persists as long as it remains an SCC in \( G_{\gamma} \). 

\begin{remark}
    [Persistence of Morse Sets]
A Morse set \( M \) has a \textbf{birth threshold} \( \gamma_b \) where it first appears and a \textbf{death threshold} \( \gamma_d \), where:
\begin{enumerate}
    \item \( M \) \textbf{merges} into another Morse set at \( \gamma_d \).
    \item The \textbf{index} of \( M \) (given by the dimension of  its \textbf{Conley index} or its homology) \textbf{changes} at \( \gamma_d \).
\end{enumerate}
Thus, the \textbf{lifespan} of a Morse set is defined as:
\[
\text{Persistence}(M) = \gamma_d - \gamma_b.
\]
\end{remark}

\begin{remark}
The change in Conley index and homology signals a \textbf{topological transition} in the Morse set structure. This can indicate a loss or gain of significant recurrent dynamics in the Markov process.
\end{remark}
 
\begin{theorem} 
Let \( \cV{\gamma_1} \) and \( \cV{\gamma_2} \) be two multivector fields such that \( \gamma_1 < \gamma_2 \). Then
for any Morse set $M_1 \in \mathcal{M}(\cV{\gamma_1})$, there exists a Morse set $M_2 \in \mathcal{M}(\cV{\gamma_2})$ such that:
    \[
     M_1\subset M_2 .
    \]
    
\end{theorem}

\begin{proof}
Since \( \gamma_1 < \gamma_2 \), the construction of \( \cV{\gamma_2} \) involves additional merging compared to \( \cV{\gamma_1} \). From that we are sure that $G_{\gamma_1}$ is actually a subgraph of $G_{\gamma_2}$ and then Any strongly connected component (SCC) in \( G_{\gamma_1} \) remains connected in \( G_{\gamma_2} \) or merges into a larger SCC.
   
\end{proof}\\

Thus, the Morse sets at \( \gamma_2 \) is a union of Morse sets at \( \gamma_1 \), establishing the filtration property.

\begin{remark}
This result shows that Morse sets \textbf{never split} as \( \gamma \) increases. Instead, they follow a strict merging hierarchy, revealing the most \textbf{persistent structures} in the system.
\end{remark}

\section{Stability of Persistence diagrams for Morse decomposition}

A fundamental property of any meaningful topological or combinatorial structure in applied mathematics is \textit{stability}: small perturbations in the input data should not lead to large variations in the extracted features. In the context of multivector fields constructed from Markov chains, stability ensures that Morse decompositions remain robust under perturbations in the transition probabilities.

The goal here is to establish a formal stability result for Morse set persistence in multivector fields. Specifically, we consider how small perturbations in the transition matrix of a Markov process affect the birth and death times of Morse sets in the associated persistence diagram. To quantify this effect, we introduce an appropriate bottleneck distance that measures the similarity between Morse set persistence diagrams obtained from different Markov chains. But before that we need to define the persistence diagram in our setting since it is a bit different from the natural one (see \cite{cohen2005stability}).

\subsection{Persistence Diagram for Morse Sets}

\begin{definition}
Let \( \mathcal{M} = \{M_i\} \) be the set of Morse sets obtained from the \textbf{M-graph} of a multivector field at different values of \( \gamma \). Each Morse set \( M_i \) is characterized by:
\begin{itemize}
    \item A \textbf{birth threshold} \( \gamma_b \), where it first appears.
    \item A \textbf{death threshold} \( \gamma_d \), where it either merges with another Morse set or undergoes an \textbf{index change}.
    \item A \textbf{dimensional topological index} \( k_i \), defined as the pair:
\end{itemize}

\begin{equation}
k_i = \left( \dim H_{1}(\overline{M_i}), \dim \text{Ind}_{1}(M_i) \right),
\end{equation}
where:
\begin{itemize}
    \item \( \dim H_{1}(\overline{M_i}) \) is the \textbf{dimension of the first homology group} of the closure \( \overline{M_i} \) of the Morse set.
    \item \( \dim \text{Ind}_{1}(M_i) \) is the \textbf{dimension of the first relative homology group of the closure of $M_i$ to its mouth}.
\end{itemize}

The \textbf{persistence diagram} consists of a set of points \( (\gamma_b, \gamma_d, k_i) \) in the \textbf{birth-death plane}, where the index \( k_i \) encodes the \textbf{topological complexity} of the Morse set across different thresholds.
\begin{remark}
    $k_i$ is not really a part of the coordinate for the persistence of our Morse set $M_i$ in our persistence diagram. It just give you informations about the nature if the Morse set $M_i$.
\end{remark}

\end{definition}

\subsection{Bottleneck Distance for Morse Set Persistence Diagrams}

To compare the persistence diagrams of Morse sets obtained from different Markov chains, we define a bottleneck distance that captures differences in their birth and death times, ensuring that the measure remains independent of any fixed threshold \( \gamma \).

\subsubsection{Preliminaries and Notation}

Let \( P \) and \( P' \) be two transition matrices representing two Markov chains on the same state space:

\[
P = (p_{ij})_{1 \leq i,j \leq n}, \quad P' = (p'_{ij})_{1 \leq i,j \leq n}.
\]

Each Markov chain induces a multivector field \( \cV{\gamma}(P) \) and \( \cV{\gamma}(P') \) at each threshold \( \gamma \), leading to the construction of an M-graph and the computation of Morse sets. The persistence diagram tracks the birth and death of each Morse set as \( \gamma \) varies.

Let:
\begin{itemize}
    \item \( D(P) = \{(\gamma_b^i, \gamma_d^i, k_i)\} \) be the persistence diagram of Morse sets obtained from \( P \).
    \item \( D(P') = \{(\gamma_b'^j, \gamma_d'^j, k_j')\} \) be the persistence diagram obtained from \( P' \).
\end{itemize}

Each point in \( D(P) \) and \( D(P') \) represents a Morse set \( M_i \) with:
\begin{itemize}
    \item \textbf{Birth time} \( \gamma_b^i \): the smallest threshold \( \gamma \) at which \( M_i \) appears.
    \item \textbf{Death time} \( \gamma_d^i \): the largest threshold \( \gamma \) at which \( M_i \) exists before merging or disappearing.
    \item \textbf{Topological index} \( k_i = (\dim H_{1}(\overline{M_i}), \dim \text{Ind}_{1}(M_i)) \), encoding the first homology and Conley index dimensions.
\end{itemize}

\subsubsection{Definition of the Bottleneck Distance}

The \textbf{bottleneck distance} between \( D(P) \) and \( D(P') \) is given by:

\begin{equation}
d_B(D(P), D(P')) = \inf_{\varphi} \sup_{i} \|\mathbf{p}_i - \mathbf{p}'_{\varphi(i)}\|_{\infty},
\end{equation}

where:
\begin{itemize}
    \item \( \varphi \) is a bijection matching points only if \( k_i = k_{\varphi(i)}' \).
    \item \( \mathbf{p}_i = (\gamma_b^i, \gamma_d^i) \) and \( \mathbf{p}'_{\varphi(i)} = (\gamma_b'^{\varphi(i)}, \gamma_d'^{\varphi(i)}) \) are the birth-death coordinates of matched Morse sets.
    \item The distance between matched pairs is computed in the \( \ell_\infty \) norm:
    \begin{equation}
    \|\mathbf{p}_i - \mathbf{p}'_{\varphi(i)}\|_{\infty} = \max \left( |\gamma_b^i - \gamma_b'^{\varphi(i)}|, |\gamma_d^i - \gamma_d'^{\varphi(i)}| \right).
    \end{equation}
    \item If a Morse set has no valid match in the other persistence diagram, it is assigned to the diagonal \( (\gamma_b = \gamma_d) \), representing a zero-persistence feature.
\end{itemize}

Unlike classical persistence in topological data analysis, where stability is typically analyzed with respect to an underlying metric space, our setting is purely combinatorial: the Markov chain defines a weighted graph, and persistence is extracted through a filtration on the multivector field induced by transition probabilities.

\subsubsection{The Stability Theorem}

\begin{theorem}[Stability of Morse Set Persistence Diagrams]
Let \( P \) and \( P' \) be two transition matrices of a Markov process with n states such that:

\[
\| P - P' \|_{\infty} < \delta.
\]

Then, the bottleneck distance between their Morse set persistence diagrams satisfies:

\begin{equation}
d_B(D(P), D(P')) < C \delta,
\end{equation}

where \( C \) depends only on the structure of the Markov process.
\end{theorem} 

Before proving the theorem, we establish a key intermediate result. Since a global perturbation \( \delta \) in \( P \) may affect multiple transitions simultaneously, we first consider the effect of a single-entry perturbation in \( P \), which will allow us to extend the result to the full transition matrix.

\begin{lemma}[Local Stability Under Single Transition Change]
Let \( P \) and \( P' \) be two transition matrices that differ at only one entry \( (i,j) \):

\[
p'_{ij} = p_{ij} + \delta, \quad \text{where} \quad |\delta| < \epsilon, \quad \text{and} \quad p_{kl} = p'_{kl} \text{ for all } (k,l) \neq (i,j).
\]

Then, the bottleneck distance between their Morse set persistence diagrams satisfies:

\begin{equation}
d_B(D(P), D(P')) < \epsilon.
\end{equation}
\end{lemma} 

\begin{proof}
Since only one transition probability \( p_{ij} \) is modified, the global structure of the multivector field remains unchanged except in one location:

\begin{itemize}
    \item The modification of \( p_{ij} \) affects only the relationship between states \( N_i \) and \( N_j \).
    \item The change in the multivector field may result in:
    \begin{itemize}
        \item A Morse set being born earlier or later.
        \item A Morse set disappearing earlier or later.
        \item A merging of multiple Morse sets into a single one.
    \end{itemize}
    \item However, in all cases, each affected Morse set undergoes a shift in birth or death time by at most \( |\delta| \).
\end{itemize}

Thus, for any affected Morse set \( M \), we have:

\[
|\gamma_b(M) - \gamma_b(M')| \leq |\delta|, \quad |\gamma_d(M) - \gamma_d(M')| \leq |\delta|.
\]

Now, we define a bijection \( \varphi \) between Morse sets in \( D(P) \) and \( D(P') \):

\begin{itemize}
    \item For all unchanged Morse sets, we match them with their identical counterparts.
    \item For affected Morse sets, we match each of them to the corresponding Morse set in \( D(P') \), whether it remains the same or has merged into a new Morse set.
\end{itemize}

Since the maximum shift in birth or death time is at most \( |\delta| \), we obtain:

\[
\|\mathbf{p}_i - \mathbf{p}'_{\varphi(i)}\|_{\infty} = \max \left( |\gamma_b(M) - \gamma_b(M')|, |\gamma_d(M) - \gamma_d(M')| \right) \leq |\delta|.
\]

By the definition of the bottleneck distance, which takes the infimum over all bijections, we conclude:

\[
d_B(D(P), D(P')) \leq \sup_{M} \max \left( |\gamma_b(M) - \gamma_b(M')|, |\gamma_d(M) - \gamma_d(M')| \right) \leq |\delta|.
\] and
\[d_B(D(P),D(P'))<\epsilon\]

\end{proof}

\subsubsection{Proof of the sability theorem}

The proof follows by decomposing the perturbation into single-entry modifications in the transition matrix and applying the local stability lemma iteratively:

\begin{itemize}
    \item \textbf{Step 1: Decomposition of Perturbation}  
    We write the transition matrix perturbation as a sequence of stepwise changes:
    \[
    P = P_0 \to P_1 \to P_2 \to \dots \to P_l = P',
    \]
    where each \( P_k \) differs from \( P_{k-1} \) in at most one transition probability.

    \item \textbf{Step 2: Application of the Local Stability Lemma}  
    From the local stability lemma, we know that if \( P \) and \( P' \) differ by a single entry change \( \delta \), then:
    \[
    d_B(D(P_k), D(P_{k+1})) \leq |\delta_k|.
    \]
    Summing over all steps, we get:
    \[
    d_B(D(P), D(P')) \leq \sum_{k=0}^{l-1} d_B(D(P_k), D(P_{k+1})) \leq l\times \max_k(\delta_k).
    \]
Since $\|P-P'\|_\infty<\delta$, we have:
\[d_B(D(P),D(P'))<l\delta
\]
$l$ being the number of non-zero entries of the matrix $P-P'$, with $l\le n^2$

\qed

\end{itemize}
\section{Example of computation}
Here we will show how the compuation of the persistence diagram for this setting looks like. We will continue with the Example we used before(Example 1).\\

We consider a Markov process with three states \( S = \{N_1, N_2, N_3\} \) and the transition matrix:

\[
P =
\begin{bmatrix}
0.5 & 0.17 & 0.33 \\
0.17 & 0.6 & 0.23 \\
0.15 & 0.15 & 0.7
\end{bmatrix}.
\]
Here the different values of gamma we are interested in are:
$\{0.15, 0.17, 0.23, 0.33\}$.

Note that we are not interested by value at the diagonal since while creating the M-graph, we always have a self loop for each multivector. 
\subsection{Multivector field construction}
As we have already shown how we construct for $\gamma =0.2$ before, I am only going to give the final construction for each $\gamma$.
\subsubsection{For$\gamma<0.15, ~~\gamma_0$}
We have: \[\cV=\{\{N_1\}, ~\{N_2\},~ \{N_3\},~\{(N_1,N_2)\},~\{(N_1,N_3)\},~\{(N_2,N_3)\}\]
Here we only have singletons since we did not cross any of the transition probabilities in the transition matrix yet

\subsubsection{For $\gamma=0.15, ~~\gamma_1$}
We have: \[\cV=\{\{N_1\}, ~\{N_2\},~ \{N_3,(N_1,N_3),(N_2,N_3)\},~\{(N_1,N_2)\}\]
\subsubsection{For $\gamma=0.17, ~~\gamma_2$}
We have: \[\cV=\{\{N_1,N_2,(N_1,N_2)\},~ \{N_3,(N_1,N_3),(N_2,N_3)\}\}\]

\subsubsection{For $\gamma=0.23, ~~\gamma_3$}
We have: \[\cV=\{\{N_1,N_2,(N_1,N_2),N_3,(N_1,N_3),(N_2,N_3)\}\}\] 

For $\gamma\ge 0.23$ which is the greastest element not in the diagonal, our multivector field will only always consist in one big multivector containing the whole space.

\subsection{Morse set computations}
Here we will also compute Morse set from the multivector field for each value of $\gamma$ and their index, which is just computing each time strongly connected component of the M-graph.
\subsubsection{For $\gamma<0.15$ }
Since we said that in the M-graph each multivector have self loop, they are then also strongly connected and then we have as Morse set for this value of $\gamma$:
\begin{itemize}
    \item \[M^0_1,~M^0_2,~M^0_3=\{N_1\},~\{N_2\},~\{N_3\}\] all of them have the same index so:
    \[k^0_1=k^0_2=k^0_3=(0,0)\], since the dimensions of the first homology and relative homology of a point are 0.
    \item \[M^0_4,~M^0_5,~M^0_6=\{(N_1,N_2)\},~\{(N_1,N_3)\},~\{(N_2,N_3)\}\] all of them have the same index so:
    \[k^0_4=k^0_5=k^0_6=(0,1)\], since the dimensions of the first homology of an edge is zero, but the first relative homology of an edge to its two vertices is actually the same as the homology of a circle.
\end{itemize}

\subsubsection{For $\gamma=0.15$ }
\begin{itemize}
    \item $M^1_1,~M^1_2=\{N_1\},~ \{N_2\}$ and $k^1_1=k^1_2=(0,0)$
    \item $M^1_3, M^1_4=\{(N_1,N_2)\},~ \{N_3,(N_1,N_3),(N_2,N_3)\}$ and $k^1_3=k^1_4=(0,1)$
 
\end{itemize}
\subsubsection{For $\gamma=0.17$ }
\begin{itemize}
    \item $M^2_1=\{N_1,N_2,(N_1,N_2)\}$ and $k^2_1=(0,0)$
    \item $M^2_2=\{N_3,(N_1,N_3),(N_2,N_3)\}$ and $k^2_2=(0,1)$
\end{itemize}
\subsubsection{For $\gamma=0.23$ }
\begin{itemize}
    \item $M^3_1=\{N_1,N_2,(N_1,N_2),N_3,(N_1,N_3),(N_2,N_3)\}$ and $k^3_1=(1,1)$
\end{itemize}
Note that the Morse set remains unchanged for greater values of $\gamma$. In general Morse sets are union of multivectors, here we just have a special case where Morse sets correspond to multivectors for the sake of simplicity.
\subsection{Persistence diagram computation}
From the computations of Morse sets we have:
\begin{itemize}
    \item $M^0_1\subset M^1_1\subset M^2_1\subset M^3_1$ \\$M^0_1$ dies at $\gamma_3$ because of the change in index.
    \item $M^0_2\subset M^1_2\subset M^2_1\subset M^3_1$ $M^0_2$ dies at $\gamma_2$ because it merges with $M^0_1$
    \item $M^0_3\subset M^1_4\subset M^2_2\subset M^3_1$\\
    $M^0_3$ dies at $\gamma_1$ because there is a change of index.
    \item $M^0_4\subset M^1_2\subset M^2_1\subset M^3_1$; $M^0_4$ dies at $\gamma_2$ because it merges with $M^0_1$

    \item $M^0_5\subset M^1_4\subset M^2_2\subset M^3_1$\\
    $M^0_5$ dies at $\gamma_3$ because there is a change of index.
    \item $M^0_6\subset M^1_4\subset M^2_2\subset M^3_1$\\
    $M^0_6$ dies at $\gamma_1$ because it merges with $M^0_5$.
    \item $M^3_1$ is born at $\gamma_3$ and never dies.

\end{itemize}
The persistence  diagram for the transition Matrix $P$ will have the points:
$p_1=(0,0.23,(0,0)),~ p_2=(0,0.17,(0,0)),~ p_3=(0,0.15,(0,0)),~p_4=(0,0.17,(0,1)),~p_5=(0,0.23,(0,1)),~p_6=(0,0.15,(0,1)),p_7=(0.23,\infty,(1,1))$
\begin{figure}
    \centering\includegraphics[height=10cm, width=10cm, scale=1.00, angle=0 ]{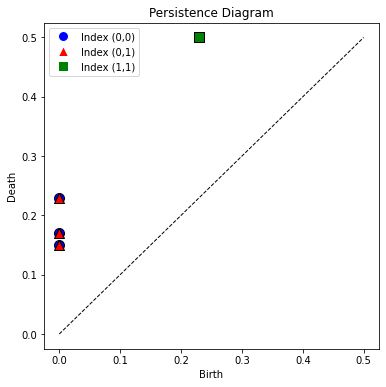}
    \caption{Persistence diagram of the Morse sets for $P$}
    \label{fig:enter-label}
\end{figure}
\section*{Competing interests}
   Research of D.W. is partially supported by the Polish National Science Center under Opus Grant No. 2019/35/B/ST1/00874.\\
    
\newpage
\bibliographystyle{plain}
\bibliography{references}

\end{document}